\documentclass[12pt]{article}%
\usepackage{amssymb,latexsym}
\usepackage{amsfonts}
\usepackage{amsmath}
\usepackage[colorlinks=true, pdfstartview=FitV, linkcolor=blue,
citecolor=blue, urlcolor=blue]{hyperref}
\usepackage{color}
\usepackage{amssymb}
\usepackage{graphicx}
\usepackage[compress]{cite}%
\providecommand{\U}[1]{\protect\rule{.1in}{.1in}}
\newtheorem{theorem}{Theorem}

\newtheorem{remark}[theorem]{Remark}

\newenvironment{proof}[1][Proof]{\noindent\textbf{#1.} }{\ \rule{0.5em}{0.5em}}
\numberwithin{equation}{section}
\numberwithin{theorem}{section}
\allowdisplaybreaks
\begin{document}

\title{Degenerate Fubini-type polynomials associated with degenerate
Apostol-Bernoulli and Apostol-Euler polynomials of order $\alpha$}
\author{Muhammet Cihat DA\u{G}LI\\Department of Mathematics, Akdeniz University,\\07058-Antalya, Turkey\\E-mail: mcihatdagli@akdeniz.edu.tr}
\date{}
\maketitle

\begin{abstract}
In this paper, by introducing the degenerate Fubini-type polynomials, we give
several relations with the help of the Fa\`{a} di Bruno formula and some
properties of Bell polynomials, and generating function methods. Also, we
derive some new explicit formulas and recurrence relations for Fubini-type
polynomials and numbers. Associating the degenerate Fubini-type polynomials
newly defined here with degenerate Apostol-Bernoulli polynomials and
degenerate Apostol-Euler polynomials of order $\alpha$ enables us to present
additional relations for some degenerate special polynomials and numbers.



\end{abstract}

\section{Introduction}

The higher-order Bernoulli polynomials $B_{n}^{(\alpha)}(x)$ and higher-order
Euler polynomials $E_{n}^{(\alpha)}(x)$, each of degree $n$ in $x$ and in
$\alpha$, are defined by means of the generating functions \cite{n}
\[
\left(  \dfrac{t}{e^{t}-1}\right)  ^{\alpha}e^{xt}=%
{\displaystyle\sum\limits_{n=0}^{\infty}}
B_{n}^{(\alpha)}(x)\dfrac{t^{n}}{n!}%
\]
and%
\[
\text{ }\left(  \dfrac{2}{e^{t}+1}\right)  ^{\alpha}e^{xt}=%
{\displaystyle\sum\limits_{n=0}^{\infty}}
E_{n}^{(\alpha)}(x)\dfrac{t^{n}}{n!},
\]
respectively. For $\alpha=1$, we have the classical Bernoulli polynomials
$B_{n}(x)$ and Euler polynomials $E_{n}(x),$ defined by means of the following
generating functions:%
\[
\dfrac{te^{xt}}{e^{t}-1}=%
{\displaystyle\sum\limits_{n=0}^{\infty}}
B_{n}(x)\dfrac{t^{n}}{n!}\text{ }\left(  \left\vert t\right\vert <2\pi\right)
\text{ and }\dfrac{2e^{xt}}{e^{t}+1}=%
{\displaystyle\sum\limits_{n=0}^{\infty}}
E_{n}(x)\dfrac{t^{n}}{n!}\text{ }\left(  \left\vert t\right\vert <\pi\right)
.
\]
In particular, the rational numbers $B_{n}=B_{n}(0)$ and integers $E_{n}%
=2^{n}E_{n}(1/2)$ are called classical Bernoulli numbers and Euler numbers, respectively.

The generalized Apostol-Bernoulli polynomials $B_{n}^{(\alpha)}(x;\gamma)$
were defined by Luo and Srivastava by means of the generating function
\cite{luo,ls,ls1}
\[
\left(  \dfrac{t}{\gamma e^{t}-1}\right)  ^{\alpha}e^{xt}=%
{\displaystyle\sum\limits_{n=0}^{\infty}}
B_{n}^{(\alpha)}(x;\gamma)\dfrac{t^{n}}{n!}%
\]%
\[
\left(  \gamma\in%
\mathbb{C}
;\text{ }\left\vert t\right\vert <2\pi\text{ if }\gamma=1\text{; }\left\vert
t\right\vert <\left\vert \log\gamma\right\vert \text{ if }\gamma\neq1\right)
,
\]
and the generalized Apostol-Euler polynomials $E_{n}^{(\alpha)}(x;\gamma)$ by
means of the generating function \cite{luo1}
\[
\text{ }\left(  \dfrac{2}{\gamma e^{t}+1}\right)  ^{\alpha}e^{xt}=%
{\displaystyle\sum\limits_{n=0}^{\infty}}
E_{n}^{(\alpha)}(x;\gamma)\dfrac{t^{n}}{n!}%
\]%
\[
\left(  \left\vert t\right\vert <\pi\text{ if }\gamma=1\text{; }\left\vert
t\right\vert <\left\vert \log(-\gamma)\right\vert \text{ if }\gamma
\neq1;\text{ }1^{\gamma}=1\right)  .
\]
Carlitz \cite{carlitz} defined degenerate Bernoulli polynomials and degenerate
Euler polynomials by%
\[
\frac{t}{\left(  1+\lambda t\right)  ^{1/\lambda}-1}\left(  1+\lambda
t\right)  ^{x/\lambda}=%
{\displaystyle\sum\limits_{n=0}^{\infty}}
B_{n}(x;\lambda)\dfrac{t^{n}}{n!},
\]
and%
\[
\frac{2}{\left(  1+\lambda t\right)  ^{1/\lambda}+1}\left(  1+\lambda
t\right)  ^{x/\lambda}=%
{\displaystyle\sum\limits_{n=0}^{\infty}}
E_{n}(x;\lambda)\dfrac{t^{n}}{n!}.
\]
For $x=0,$ these are called as degenerate Bernoulli and Euler numbers.

The degenerate versions of Apostol-Bernoulli polynomials and Apostol-Euler
polynomials of order $\alpha$ were introduced by \cite{khan}%
\begin{equation}
\left(  \frac{t}{\gamma\left(  1+\lambda t\right)  ^{1/\lambda}-1}\right)
^{\alpha}\left(  1+\lambda t\right)  ^{x/\lambda}=%
{\displaystyle\sum\limits_{n=0}^{\infty}}
B_{n}^{\left(  \alpha\right)  }(x;\lambda;\gamma)\dfrac{t^{n}}{n!} \label{19}%
\end{equation}
and%
\begin{equation}
\left(  \frac{2}{\gamma\left(  1+\lambda t\right)  ^{1/\lambda}+1}\right)
^{\alpha}\left(  1+\lambda t\right)  ^{x/\lambda}=%
{\displaystyle\sum\limits_{n=0}^{\infty}}
E_{n}^{(\alpha)}(x;\lambda;\gamma)\dfrac{t^{n}}{n!}, \label{21}%
\end{equation}
respectively. Note that since $\lim_{\lambda\rightarrow0}\left(  1+\lambda
t\right)  ^{1/\lambda}=e^{t}$, for $\lambda\rightarrow0,$ $\alpha=\gamma=1,$
the equations (\ref{19}) and (\ref{21}) reduce to the generating functions for
classical Bernoulli and Euler polynomials, respectively.

Let us mention that the above polynomials have been discussed detailed in the
literature. (See for example \cite{dss,has,hasa,lu,luo1} and related
references therein).

We now focus on K\i lar and Simsek's recent study \cite{ks}, in which a family
of Fubini-type polynomials $a_{n}^{\left(  \alpha\right)  }\left(  x\right)  $
are introduced as in the following%
\begin{equation}
\frac{2^{\alpha}}{\left(  2-e^{t}\right)  ^{2\alpha}}e^{xt}=%
{\displaystyle\sum\limits_{n=0}^{\infty}}
a_{n}^{\left(  \alpha\right)  }\left(  x\right)  \frac{t^{n}}{n!},\text{
\ \ }\alpha\in%
\mathbb{N}
_{0}\text{ and }\left\vert t\right\vert <\log2. \label{0}%
\end{equation}
In particular, $a_{n}^{\left(  \alpha\right)  }\left(  0\right)
=a_{n}^{\left(  \alpha\right)  }$ are called Fubini-type numbers. They gave
some relationships between these polynomials and numbers, and other celebrated
polynomials and numbers such as Apostol-Bernoulli numbers, the Frobenius-Euler
numbers and the Stirling numbers via generating function methods and
functional equations. Very recently, Srivastava and K\i z\i lates \cite{sk}
extended Fubini-type polynomials $a_{n}^{\left(  \alpha\right)  }\left(
x\right)  $ to parametric kind families of the Fubini-type polynomials by
considering the two special generating functions and obtained many relations
concerning these and other parametric special polynomials and numbers. As
emphasized therein, the Fubini-type polynomials $a_{n}^{\left(  \alpha\right)
}\left(  x\right)  $ are special case of generalized Apostol-Euler polynomials
$E_{n}^{(\alpha)}(x;\gamma).$ More concretely, $E_{n}^{(2\alpha)}%
(x;-1/2)=2^{3\alpha}a_{n}^{\left(  \alpha\right)  }\left(  x\right)  .$

Further investigations on Fubini polynomials and numbers can be found in
\cite{kargin,ks1,kim,qi,skr,su}, and plenty of references cited therein.

On the other hand, Qi and his colleagues have studied a number of explicit and
recursive formulas, and closed forms for some significant polynomials and
numbers by applying the Fa\`{a} di Bruno formula (see Eq. (\ref{2}), below),
some properties of the Bell polynomials of the second kind, and a general
derivative formula for a ratio of two differentiable functions. See
\cite{dagli,dagli1,d,guo,hu-kim,qi1,qi3,qi9,qi4,qi2,qi5,qi6,qi15,qi7,qi12,wei}
and related references.

In this paper, we introduce degenerate version of Fubini-type polynomials as%
\begin{equation}
\frac{2^{\alpha}}{\left(  2-\left(  1+\lambda t\right)  ^{1/\lambda}\right)
^{2\alpha}}\left(  1+\lambda t\right)  ^{x/\lambda}=%
{\displaystyle\sum\limits_{n=0}^{\infty}}
a_{n}^{\left(  \alpha\right)  }\left(  x;\lambda\right)  \frac{t^{n}}%
{n!},\text{ \ \ \ }\lambda\in%
\mathbb{R}
. \label{1}%
\end{equation}
Notice that for $x=0,$ $a_{n}^{\left(  \alpha\right)  }\left(  0;\lambda
\right)  =a_{n}^{\left(  \alpha\right)  }\left(  \lambda\right)  $ are called
degenerate Fubini-type numbers. Also, for $\lambda\rightarrow0$, these reduce
to Fubini-type polynomials $a_{n}^{\left(  \alpha\right)  }\left(  x\right)  $
aforementioned above.

In parallel with the conclusion given in \cite[Remark 4]{sk}, we infer a
relationship between degenerate Fubini-type polynomials and degenerate
Apostol-Euler polynomials of order $\alpha,$ i.e.%
\begin{equation}
a_{n}^{\left(  \alpha\right)  }\left(  x;\lambda\right)  =2^{-3\alpha}%
E_{n}^{(2\alpha)}(x;\lambda;-1/2). \label{22}%
\end{equation}

In this paper, we would like to use Fa\`{a} di Bruno formula and some
properties of Bell polynomials, and generating function methods in order to
obtain some new explicit formulas, closed forms and recurrence relations for
degenerate Fubini-type polynomials and numbers, and Fubini-type polynomials
and numbers. Moreover, we give a relation between degenerate Fubini-type
polynomials and degenerate Apostol-Bernoulli polynomials of order $\alpha$ and
deduce similar formulas for them.

\section{Properties of second kind Bell polynomials}

The Bell polynomials of the second kind $B_{n,k}\left(  x_{1},x_{2}%
,...,x_{n-k+1}\right)  $ for $n\geq k\geq0$ were defined by \cite[p. 134 and
139]{Comtet}
\[
B_{n,k}\left(  x_{1},x_{2},...,x_{n-k+1}\right)  =%
{\displaystyle\sum\limits_{\substack{1\leq i\leq n,\text{ }l_{i}\in\left\{
0\right\}  \cup\mathbb{N}\\{{\textstyle\sum\nolimits_{i=1}^{n}}}%
il_{i}=n,\text{ }{{\textstyle\sum\nolimits_{i=1}^{n}}}l_{i}=k}}^{\infty}}
\frac{n!}{%
{\textstyle\prod\nolimits_{i=1}^{l-k+1}}
l_{i}!}%
{\textstyle\prod\limits_{i=1}^{l-k+1}}
\left(  \frac{x_{i}}{i!}\right)  ^{l_{i}}.
\]
The Fa\`{a} di Bruno formula can be described in terms of the Bell polynomials
of the second kind $B_{n,k}\left(  x_{1},x_{2},...,x_{n-k+1}\right)  $ by%
\begin{equation}
\frac{d^{n}}{dt^{n}}f\circ h\left(  t\right)  =%
{\displaystyle\sum\limits_{k=0}^{n}}
f^{\left(  k\right)  }\left(  h\left(  t\right)  \right)  B_{n,k}\left(
h^{\prime}\left(  t\right)  ,h^{\prime\prime}\left(  t\right)
,...,h^{(n-k+1)}\left(  t\right)  \right)  . \label{2}%
\end{equation}

For $n\geq k\geq0,$ these polynomials satisfy the following relation \cite[p.
135]{Comtet}%
\begin{equation}
B_{n,k}\left(  abx_{1},ab^{2}x_{2},...,ab^{n-k+1}x_{n-k+1}\right)  =a^{k}%
b^{n}B_{n,k}\left(  x_{1},x_{2},...,x_{n-k+1}\right)  , \label{5}%
\end{equation}
where $a$ and $b$ are any complex number. Also, for $n\geq k\geq0,$ the
following formula is valid for the special case of $B_{n,k}$%
\begin{equation}
B_{n,k}\left(  1,1,...,1\right)  =S\left(  n,k\right)  , \label{12}%
\end{equation}
where $S\left(  n,k\right)  $ denotes the Stirling numbers of the second kind,
can be generated by \cite[p. 206]{Comtet}
\[
\frac{\left(  e^{t}-1\right)  ^{k}}{k!}=%
{\displaystyle\sum\limits_{n=k}^{\infty}}
S\left(  n,k\right)  \frac{t^{n}}{n!}.
\]

In \cite[Remark 1]{qi80}, there existed the formula%

\begin{equation}
B_{n,k}\left(  1,1-\lambda,\left(  1-\lambda\right)  \left(  1-2\lambda
\right)  ,...,\prod_{l=0}^{n-k}\left(  1-l\lambda\right)  \right)
=\frac{(-1)^{k}}{k!}%
{\displaystyle\sum\limits_{l=0}^{k}}
\left(  -1\right)  ^{l}\binom{k}{l}\prod_{q=0}^{n-1}\left(  l-q\lambda\right)
, \label{14}%
\end{equation}
which is equivalent to%
\begin{equation}
B_{n,k}\left(  \left\langle \lambda\right\rangle _{1},\left\langle
\lambda\right\rangle _{2},...,\left\langle \lambda\right\rangle _{n-k+1}%
\right)  =\frac{(-1)^{k}}{k!}%
{\displaystyle\sum\limits_{l=0}^{k}}
\left(  -1\right)  ^{l}\binom{k}{l}\left\langle \lambda l\right\rangle _{n},
\label{15}%
\end{equation}
established in \cite[Theorems 2.1 and 4.1]{qi10}. In \cite[Remark 7.5]{guo1},
the explicit formulas (\ref{14}) and (\ref{15}) were rearranged as%
\begin{align*}
&  B_{n,k}\left(  1,1-\lambda,\left(  1-\lambda\right)  \left(  1-2\lambda
\right)  ,...,\prod_{l=0}^{n-k}\left(  1-l\lambda\right)  \right) \\
&  =(-1)^{k}\frac{\lambda^{n-1}(n-1)!}{k!}%
{\displaystyle\sum\limits_{l=1}^{k}}
\left(  -1\right)  ^{l}l\binom{k}{l}\binom{l/\lambda-1}{n-1}%
\end{align*}
for $\lambda\neq0$ and
\begin{equation}
B_{n,k}\left(  \left\langle \lambda\right\rangle _{1},\left\langle
\lambda\right\rangle _{2},...,\left\langle \lambda\right\rangle _{n-k+1}%
\right)  =(-1)^{k}\lambda\frac{(n-1)!}{k!}%
{\displaystyle\sum\limits_{l=1}^{k}}
\left(  -1\right)  ^{l}l\binom{k}{l}\binom{\lambda l-1}{n-1}. \label{17}%
\end{equation}
Here, the generalized binomial coefficient $\binom{z}{w}$ is defined by%
\[
\binom{z}{w}=%
\begin{cases}
\dfrac{\Gamma\left(  z+1\right)  }{\Gamma\left(  w+1\right)  \Gamma\left(
z-w+1\right)  }, & \text{ if }z,w,z-w\in%
\mathbb{C}
-\left\{  -1,-2,...\right\}  \text{;}\\
0, & \text{if }z\in%
\mathbb{C}
-\left\{  -1,-2,...\right\}  \text{ and }w,z-w\in\left\{  -1,-2,...\right\}
\text{.}%
\end{cases}
\]

\section{Results and their proofs}

In this section, we give some computational formulas for degenerate
Fubini-type numbers, some explicit formulas and recurrence relations for
Fubini-type polynomials and numbers and consequently, degenerate
Apostol-Bernoulli polynomials and degenerate Apostol-Euler polynomials of
order $\alpha.$ Also, we present further relations for some polynomials,
considered here.

\begin{theorem}
The degenerate Fubini-type numbers can be computed by the formula:%
\begin{equation}
a_{n}^{\left(  \alpha\right)  }\left(  \lambda\right)  =(n-1)!%
{\displaystyle\sum\limits_{k=1}^{n}}
\frac{\left\langle -2\alpha\right\rangle _{k}}{2^{\alpha+k}}\frac{(-1)^{k}%
}{\lambda^{k-1}k!}%
{\displaystyle\sum\limits_{l=1}^{k}}
(-1)^{l}l\binom{k}{l}\binom{\lambda l-1}{n-1}, \label{23}%
\end{equation}
where $\left\langle x\right\rangle _{n}$ denotes the falling factorial,
defined for $x\in%
\mathbb{R}
$ by
\[
\left\langle x\right\rangle _{n}=\prod\limits_{k=0}^{n-1}\left(  x-k\right)  =%
\begin{cases}
x\left(  x-1\right)  ...(x-n+1), & \text{if }n\geq1\text{;}\\
1, & \text{if }n=0.
\end{cases}
\]
Consequently, for the special case of the degenerate Apostol-Euler polynomials
$E_{n}^{(\alpha)}(x;\lambda;\gamma),$ the following relation holds:%
\begin{equation}
E_{n}^{(2\alpha)}(0;\lambda;-1/2)=(n-1)!%
{\displaystyle\sum\limits_{k=1}^{n}}
\frac{\left\langle -2\alpha\right\rangle _{k}}{2^{k-2\alpha}}\frac{(-1)^{k}%
}{\lambda^{k-1}k!}%
{\displaystyle\sum\limits_{l=1}^{k}}
(-1)^{l}l\binom{k}{l}\binom{\lambda l-1}{n-1}. \label{24}%
\end{equation}

\end{theorem}

\begin{proof}
If we apply $f(u)=(2-u)^{-2\alpha}$ and $u=g(t)=\left(  1+\lambda t\right)
^{1/\lambda}$ to the Fa\`{a} di Bruno formula (\ref{2}), and use (\ref{5}) and
(\ref{17}) then, we find that%
\begin{align}
&  \frac{d^{n}}{dt^{n}}\left(  (2-\left(  1+\lambda t\right)  ^{1/\lambda
})^{-2\alpha}\right) \nonumber\\
&  =%
{\displaystyle\sum\limits_{k=0}^{n}}
\frac{d^{k}}{du^{k}}(2-u)^{-2\alpha}B_{n,k}\left(  \frac{\lambda\left(
1+t\right)  ^{\lambda-1}}{\lambda},\frac{\lambda\left(  \lambda-1\right)
\left(  1+t\right)  ^{\lambda-2}}{\lambda},\right. \nonumber\\
&  ...,\left.  \frac{\lambda\left(  \lambda-1\right)  ...\left(
\lambda-\left(  n-k\right)  \right)  \left(  1+t\right)  ^{\lambda-(n-k+1)}%
}{\lambda}\right)  ,\nonumber\\
&  =%
{\displaystyle\sum\limits_{k=0}^{n}}
\left\langle -2\alpha\right\rangle _{k}(2-u)^{-2\alpha-k}\frac{\left(
1+t\right)  ^{k\lambda-n}}{\lambda^{k}}B_{n,k}\left(  \left\langle
\alpha\right\rangle _{1},\left\langle \alpha\right\rangle _{2}%
,...,\left\langle \alpha\right\rangle _{n-k+1}\right) \nonumber\\
&  =%
{\displaystyle\sum\limits_{k=0}^{n}}
\left\langle -2\alpha\right\rangle _{k}(2-u)^{-2\alpha-k}\frac{\left(
1+t\right)  ^{k\lambda-n}}{\lambda^{k}}(-1)^{k}\lambda\frac{(n-1)!}{k!}%
{\displaystyle\sum\limits_{l=1}^{k}}
\left(  -1\right)  ^{l}l\binom{k}{l}\binom{\lambda l-1}{n-1}. \label{20}%
\end{align}
Now letting $t\rightarrow0$, which is equivalent to $u\rightarrow0$ on both
sides of (\ref{20}) and taking into consideration the generating function for
degenerate Fubini-type numbers (for $x=0$ in equation (\ref{1})) complete the
proof of (\ref{23}). From the relationship (\ref{22}), the identity (\ref{24})
follows readily.
\end{proof}

\begin{theorem}
The Fubini-type polynomials $a_{n}^{\left(  \alpha\right)  }\left(  x\right)
$ possess the explicit formula
\[
a_{n}^{\left(  \alpha\right)  }\left(  x\right)  =2^{\alpha}%
{\displaystyle\sum\limits_{k=0}^{n}}
\binom{n}{k}%
{\displaystyle\sum\limits_{i=0}^{k}}
\left\langle -2\alpha\right\rangle _{i}\left(  -1\right)  ^{i}S(k,i)x^{n-k},
\]
where $S(n,k)$ is the Stirling numbers of the second kind. Also, the
Fubini-type numbers $a_{n}^{\left(  \alpha\right)  }$ can be written in the
form%
\begin{equation}
a_{n}^{\left(  \alpha\right)  }=2^{\alpha}%
{\displaystyle\sum\limits_{i=0}^{n}}
\left\langle -2\alpha\right\rangle _{i}\left(  -1\right)  ^{i}S(n,i).
\label{11}%
\end{equation}
Besides, the generalized Apostol-Euler numbers $E_{n}^{(\alpha)}(x;\gamma)$
can be expressed as
\begin{equation}
E_{n}^{(2\alpha)}(x;-1/2)=2^{4\alpha}%
{\displaystyle\sum\limits_{k=0}^{n}}
\binom{n}{k}%
{\displaystyle\sum\limits_{i=0}^{k}}
\left\langle -2\alpha\right\rangle _{i}\left(  -1\right)  ^{i}S(k,i)x^{n-k}.
\label{25}%
\end{equation}

\end{theorem}

\begin{proof}
From (\ref{2}), (\ref{5}) and (\ref{12}), we have%
\begin{align}
&  \frac{d^{k}}{dt^{k}}\left(  2-e^{t}\right)  ^{-2\alpha}\nonumber\\
&  =%
{\displaystyle\sum\limits_{i=0}^{k}}
\left\langle -2\alpha\right\rangle _{i}\left(  2-e^{t}\right)  ^{-2\alpha
-i}B_{k,i}\left(  -e^{t},-e^{t},...,-e^{t}\right) \nonumber\\
&  =%
{\displaystyle\sum\limits_{i=0}^{k}}
\left\langle -2\alpha\right\rangle _{i}\left(  2-e^{t}\right)  ^{-2\alpha
-i}(-1)^{i}e^{ti}B_{k,i}\left(  1,1,...,1\right) \nonumber\\
&  \rightarrow%
{\displaystyle\sum\limits_{i=0}^{k}}
\left\langle -2\alpha\right\rangle _{i}(-1)^{i}S\left(  k,i\right)  ,\text{
\ \ \ \ as }t\rightarrow0. \label{29}%
\end{align}
Also, it is obvious that $\left(  e^{xt}\right)  ^{(k)}=x^{k}e^{xt}\rightarrow
x^{k},$ as $t\rightarrow0.$ So, by aid of the Leibnitz's formula for the $n$th
derivative of the product of two functions, we get
\begin{align*}
&  \lim_{t\rightarrow0}\frac{d^{n}}{dt^{n}}\left[  \frac{2^{\alpha}}{\left(
2-e^{t}\right)  ^{2\alpha}}e^{xt}\right] \\
&  =2^{\alpha}%
{\displaystyle\sum\limits_{k=0}^{n}}
\binom{n}{k}%
{\displaystyle\sum\limits_{i=0}^{k}}
\left\langle -2\alpha\right\rangle _{i}(-1)^{i}S\left(  k,i\right)  x^{n-k},
\end{align*}
namely, we have $a_{n}^{\left(  \alpha\right)  }\left(  x\right)  $ by
(\ref{0}). For $x=0,$ we immediately arrive at the identity (\ref{11}). The
equation (\ref{25}) can be deduced from the relation between the Fubini-type
polynomials $a_{n}^{\left(  \alpha\right)  }\left(  x\right)  $ and
generalized Apostol-Euler polynomials $E_{n}^{(\alpha)}(x;\gamma),$ given by
(\ref{22}).
\end{proof}

\begin{theorem}
The Fubini-type polynomials $a_{n}^{\left(  \alpha\right)  }\left(  x\right)
$ satisfy the recurrence relation
\[%
{\displaystyle\sum\limits_{k=0}^{n}}
\binom{n}{k}%
{\displaystyle\sum\limits_{i=0}^{n-k}}
\left\langle 2\alpha\right\rangle _{i}(-1)^{i}S\left(  n-k,i\right)
a_{k}^{\left(  \alpha\right)  }\left(  x\right)  =2^{\alpha}x^{n}%
\]
In particular, the Fubini-type numbers $a_{n}^{\left(  \alpha\right)  }$
provide that%
\begin{equation}%
{\displaystyle\sum\limits_{k=0}^{n}}
\binom{n}{k}%
{\displaystyle\sum\limits_{i=0}^{n-k}}
\left\langle 2\alpha\right\rangle _{i}(-1)^{i}S\left(  n-k,i\right)
a_{k}^{\left(  \alpha\right)  }=0. \label{16}%
\end{equation}
In analogy, the generalized Apostol-Euler polynomials $E_{n}^{(\alpha
)}(x;\gamma)$ possess the recurrence relation%
\begin{equation}%
{\displaystyle\sum\limits_{k=0}^{n}}
\binom{n}{k}%
{\displaystyle\sum\limits_{i=0}^{n-k}}
\left\langle 2\alpha\right\rangle _{i}(-1)^{i}S\left(  n-k,i\right)
E_{k}^{(2\alpha)}(x;-1/2)=2^{4\alpha}x^{n}. \label{26}%
\end{equation}

\end{theorem}

\begin{proof}
Since%
\[
\left[  \left(  2-e^{t}\right)  ^{2\alpha}\right]  \left[  \frac{2^{\alpha}%
}{\left(  2-e^{t}\right)  ^{2\alpha}}e^{xt}\right]  =2^{\alpha}e^{xt},
\]
by keeping in mind the generating function of Fubini-type polynomials
(\ref{0}) and by proceeding as in the proof of (\ref{29}), differentiate $n$
times with respect to $t$ on both sides to deduce that%
\begin{align*}
&
{\displaystyle\sum\limits_{k=0}^{n}}
\binom{n}{k}\frac{\partial^{n-k}}{\partial t^{n-k}}\left[  \left(
2-e^{t}\right)  ^{2\alpha}\right]  \frac{\partial^{k}}{\partial t^{k}}\left[
\frac{2^{\alpha}}{\left(  2-e^{t}\right)  ^{2\alpha}}e^{xt}\right] \\
&  =%
{\displaystyle\sum\limits_{k=0}^{n}}
\binom{n}{k}%
{\displaystyle\sum\limits_{i=0}^{n-k}}
\left\langle 2\alpha\right\rangle _{i}\left(  2-e^{t}\right)  ^{2\alpha
-i}(-1)^{i}e^{ti}S\left(  n-k,i\right)  \frac{\partial^{k}}{\partial u^{k}%
}\left[  \frac{2^{\alpha}}{\left(  2-e^{t}\right)  ^{2\alpha}}e^{xt}\right] \\
&  \rightarrow%
{\displaystyle\sum\limits_{k=0}^{n}}
\binom{n}{k}%
{\displaystyle\sum\limits_{i=0}^{n-k}}
\left\langle 2\alpha\right\rangle _{i}(-1)^{i}S\left(  n-k,i\right)
a_{k}^{\left(  \alpha\right)  }\left(  x\right)  ,\text{ \ \ as }%
t\rightarrow0\\
&  =2^{\alpha}x^{n}.
\end{align*}
Setting $x=0$ yields the equation (\ref{16}) immediately. Formula (\ref{26})
can be derived by the same motivation stated in the proofs of our previous theorems.
\end{proof}

\begin{theorem}
The following relationship holds true:%
\begin{equation}
a_{n-2\alpha}^{\left(  \alpha\right)  }\left(  x;\lambda\right)  =\frac
{B_{n}^{\left(  2\alpha\right)  }(x;\lambda;1/2)}{2^{\alpha}\left\langle
n\right\rangle _{2\alpha}}, \label{27}%
\end{equation}
where $B_{n}^{\left(  \alpha\right)  }(x;\lambda;\gamma)$ is the degenerate
Apostol-Bernoulli polynomials of order $\alpha,$ defined by (\ref{19}).
\end{theorem}

\begin{proof}
If we put $\gamma=1/2$ and replace $\alpha$ by $2\alpha$ in (\ref{19}), we
have%
\begin{align*}%
{\displaystyle\sum\limits_{n=0}^{\infty}}
B_{n}^{\left(  2\alpha\right)  }(x;\lambda;1/2)\dfrac{t^{n}}{n!}  &  =\left(
\frac{t}{\dfrac{1}{2}\left(  1+\lambda t\right)  ^{1/\lambda}-1}\right)
^{2\alpha}\left(  1+\lambda t\right)  ^{x/\lambda}\\
&  =2^{\alpha}t^{2\alpha}%
{\displaystyle\sum\limits_{n=0}^{\infty}}
a_{n}^{\left(  \alpha\right)  }\left(  x;\lambda\right)  \frac{t^{n}}{n!}\\
&  =2^{\alpha}%
{\displaystyle\sum\limits_{n=2\alpha}^{\infty}}
a_{n-2\alpha}^{\left(  \alpha\right)  }\left(  x;\lambda\right)  \frac{t^{n}%
}{(n-2\alpha)!}\\
&  =2^{\alpha}%
{\displaystyle\sum\limits_{n=2\alpha}^{\infty}}
\left\langle n\right\rangle _{2\alpha}a_{n-2\alpha}^{\left(  \alpha\right)
}\left(  x;\lambda\right)  \frac{t^{n}}{n!},
\end{align*}
which concludes the proof.
\end{proof}

Let us continue to study degenerate Fubini-type polynomials $a_{n}^{\left(
\alpha\right)  }\left(  x;\lambda\right)  $. Firstly, from the fact%
\begin{align*}
\left(  1+\lambda t\right)  ^{x/\lambda}  &  =%
{\displaystyle\sum\limits_{n=0}^{\infty}}
\left\langle \frac{x}{\lambda}\right\rangle _{n}\lambda^{n}\frac{t^{n}}{n!}\\
&  =%
{\displaystyle\sum\limits_{n=0}^{\infty}}
\left(  x\right)  _{n,\lambda}\frac{t^{n}}{n!},
\end{align*}
where $\left(  x\right)  _{n,\lambda}=x\left(  x-\lambda\right)  ...\left(
x-\left(  n-1\right)  \lambda\right)  $ for $n>0$ with $\left(  x\right)
_{0,\lambda}=1,$ it is easily verify that%
\begin{align*}%
{\displaystyle\sum\limits_{n=0}^{\infty}}
a_{n}^{\left(  \alpha\right)  }\left(  x;\lambda\right)  \frac{t^{n}}{n!}  &
=\frac{2^{\alpha}}{\left(  2-\left(  1+\lambda t\right)  ^{1/\lambda}\right)
^{2\alpha}}\left(  1+\lambda t\right)  ^{x/\lambda}\\
&  =\left(
{\displaystyle\sum\limits_{n=0}^{\infty}}
a_{n}^{\left(  \alpha\right)  }\left(  \lambda\right)  \frac{t^{n}}%
{n!}\right)  \left(
{\displaystyle\sum\limits_{n=0}^{\infty}}
\left(  x\right)  _{n,\lambda}\frac{t^{n}}{n!}\right) \\
&  =%
{\displaystyle\sum\limits_{n=0}^{\infty}}
\left(
{\displaystyle\sum\limits_{k=0}^{n}}
\binom{n}{k}a_{k}^{\left(  \alpha\right)  }\left(  \lambda\right)  \left(
x\right)  _{n-k,\lambda}\right)  \frac{t^{n}}{n!}.
\end{align*}
Comparing the coefficients $\frac{t^{n}}{n!}$ gives the following theorem.

\begin{theorem}
For $n\geq0,$ we have%
\[
a_{n}^{\left(  \alpha\right)  }\left(  x;\lambda\right)  =%
{\displaystyle\sum\limits_{k=0}^{n}}
\binom{n}{k}a_{k}^{\left(  \alpha\right)  }\left(  \lambda\right)  \left(
x\right)  _{n-k,\lambda}.
\]

\end{theorem}

Now, we observe that%
\begin{align*}
&
{\displaystyle\sum\limits_{n=0}^{\infty}}
\left(  a_{n}^{\left(  \alpha\right)  }\left(  x+1;\lambda\right)
-a_{n}^{\left(  \alpha\right)  }\left(  x;\lambda\right)  \right)  \frac
{t^{n}}{n!}\\
&  =\frac{2^{\alpha}\left(  1+\lambda t\right)  ^{x/\lambda}}{\left(
2-\left(  1+\lambda t\right)  ^{1/\lambda}\right)  ^{2\alpha}}\left(  \left(
1+\lambda t\right)  ^{1/\lambda}-1\right) \\
&  =\frac{2^{\alpha}\left(  1+\lambda t\right)  ^{x/\lambda}}{\left(
2-\left(  1+\lambda t\right)  ^{1/\lambda}\right)  ^{2\alpha-1}}\left(
-1+\frac{1}{2-\left(  1+\lambda t\right)  ^{1/\lambda}}\right) \\
&  =\frac{2^{\alpha}\left(  1+\lambda t\right)  ^{x/\lambda}}{\left(
2-\left(  1+\lambda t\right)  ^{1/\lambda}\right)  ^{2\alpha}}-\sqrt{2}%
\frac{2^{\alpha-1/2}\left(  1+\lambda t\right)  ^{x/\lambda}}{\left(
2-\left(  1+\lambda t\right)  ^{1/\lambda}\right)  ^{2\left(  \alpha
-1/2\right)  }}\\
&  =%
{\displaystyle\sum\limits_{n=0}^{\infty}}
a_{n}^{\left(  \alpha\right)  }\left(  x;\lambda\right)  \frac{t^{n}}%
{n!}-\sqrt{2}%
{\displaystyle\sum\limits_{n=0}^{\infty}}
a_{n}^{\left(  \alpha-1/2\right)  }\left(  x;\lambda\right)  \frac{t^{n}}%
{n!}\\
&  =%
{\displaystyle\sum\limits_{n=0}^{\infty}}
\left(  a_{n}^{\left(  \alpha\right)  }\left(  x;\lambda\right)  -\sqrt
{2}a_{n}^{\left(  \alpha-1/2\right)  }\left(  x;\lambda\right)  \right)
\frac{t^{n}}{n!}.
\end{align*}
Comparing the coefficients $\frac{t^{n}}{n!}$ yields the following theorem.

\begin{theorem}
\label{main1}For $n\geq0,$ we have%
\[
a_{n}^{\left(  \alpha\right)  }\left(  x+1;\lambda\right)  =2a_{n}^{\left(
\alpha\right)  }\left(  x;\lambda\right)  -\sqrt{2}a_{n}^{\left(
\alpha-1/2\right)  }\left(  x;\lambda\right)  .
\]

\end{theorem}

Now, if we differentiate both sides of (\ref{1}) with respect to $t,$ we write%
\begin{align}
&  \frac{d}{dt}\left(  \frac{2^{\alpha}\left(  1+\lambda t\right)
^{x/\lambda}}{\left(  2-\left(  1+\lambda t\right)  ^{1/\lambda}\right)
^{2\alpha}}\right) \nonumber\\
&  =2^{\alpha}\left(  \frac{x\left(  1+\lambda t\right)  ^{(x-\lambda
)/\lambda}}{\left(  2-\left(  1+\lambda t\right)  ^{1/\lambda}\right)
^{2\alpha}}+2\alpha\frac{\left(  1+\lambda t\right)  ^{(x-\lambda+1)/\lambda}%
}{\left(  2-\left(  1+\lambda t\right)  ^{1/\lambda}\right)  ^{2\alpha+1}%
}\right)  . \label{3}%
\end{align}
If we replace $x$ by $x_{1}+x_{2}+\lambda$ and evaluate the terms on both
sides of (\ref{3}), separately, then, we have%
\begin{equation}
\frac{d}{dt}\left(  \frac{2^{\alpha}\left(  1+\lambda t\right)  ^{\left(
x_{1}+x_{2}+\lambda\right)  /\lambda}}{\left(  2-\left(  1+\lambda t\right)
^{1/\lambda}\right)  ^{2\alpha}}\right)  =%
{\displaystyle\sum\limits_{n=0}^{\infty}}
a_{n+1}^{\left(  \alpha\right)  }\left(  x_{1}+x_{2}+\lambda;\lambda\right)
\frac{t^{n}}{n!}, \label{6}%
\end{equation}%
\begin{equation}
\left(  x_{1}+x_{2}+\lambda\right)  \frac{2^{\alpha}\left(  1+\lambda
t\right)  ^{(x_{1}+x_{2})/\lambda}}{\left(  2-\left(  1+\lambda t\right)
^{1/\lambda}\right)  ^{2\alpha}}=\left(  x_{1}+x_{2}+\lambda\right)
{\displaystyle\sum\limits_{n=0}^{\infty}}
a_{n}^{\left(  \alpha\right)  }\left(  x_{1}+x_{2};\lambda\right)  \frac
{t^{n}}{n!} \label{7}%
\end{equation}
and%
\begin{equation}
\sqrt{2}\alpha\frac{2^{\alpha+1/2}\left(  1+\lambda t\right)  ^{(x_{1}%
+x_{2}+1)/\lambda}}{\left(  2-\left(  1+\lambda t\right)  ^{1/\lambda}\right)
^{2\left(  \alpha+1/2\right)  }}=\sqrt{2}\alpha%
{\displaystyle\sum\limits_{n=0}^{\infty}}
a_{n}^{\left(  \alpha+1/2\right)  }\left(  x_{1}+x_{2}+1;\lambda\right)
\frac{t^{n}}{n!}. \label{8}%
\end{equation}
Substitute (\ref{6}), (\ref{7}) and (\ref{8}) in (\ref{3}) to reach the
following theorem.

\begin{theorem}
\label{main}For $n\geq0,$ the degenerate Fubini-type polynomials
$a_{n}^{\left(  \alpha\right)  }\left(  x;\lambda\right)  $ satisfy the
recurrence relation%
\begin{align}
&  a_{n+1}^{\left(  \alpha\right)  }\left(  x_{1}+x_{2}+\lambda;\lambda\right)
\nonumber\\
&  =\left(  x_{1}+x_{2}+\lambda\right)  a_{n}^{\left(  \alpha\right)  }\left(
x_{1}+x_{2};\lambda\right)  +\sqrt{2}\alpha a_{n}^{\left(  \alpha+1/2\right)
}\left(  x_{1}+x_{2}+1;\lambda\right)  . \label{18}%
\end{align}
Letting $\lambda\rightarrow0$ and taking $x_{1}+x_{2}=y$ in (\ref{18}) allow
us to derive the following formula for Fubini-type polynomials
\[
a_{n+1}^{\left(  \alpha\right)  }\left(  y\right)  =ya_{n}^{\left(
\alpha\right)  }\left(  y\right)  +\sqrt{2}\alpha a_{n}^{\left(
\alpha+1/2\right)  }\left(  y+1\right)  .
\]

\end{theorem}

\begin{remark}
From the relationships (\ref{27}) and (\ref{22}), the counterpart identities
in Theorems \ref{main1} and \ref{main} can be presented for degenerate
Apostol-Bernoulli polynomials and degenerate Apostol-Euler polynomials of
order $\alpha.$
\end{remark}

\section{Conclusion}

In our recent study, we have introduced and dealt with degenerate version of
Fubini-type polynomials. Utilizing the Fa\`{a} di Bruno formula and some
properties of Bell polynomials, and generating function methods, we have
derived some new explicit formulas, closed forms and recurrence relations for
degenerate Fubini-type polynomials and numbers, and Fubini-type polynomials
and numbers, defined by K\i lar and Simsek \cite{ks}. Furthermore, by
associating the degenerate Fubini-type polynomials with degenerate
Apostol-Bernoulli polynomials and degenerate Apostol-Euler polynomials of
order $\alpha,$ we have presented analog identities for them. As a final note,
a relation involving degenerate Fubini-type polynomials and degenerate
Apostol-Genocchi polynomials of order $\alpha$, defined by \cite[Eq.
2.6]{khan} can be given and further relations can be obtained.

\end{document}